\documentclass[12pt]{amsart}
\usepackage{graphics}
\usepackage{setspace}
\usepackage{amsfonts,amssymb,color}
\usepackage[mathscr]{eucal}
\usepackage{amsmath, amsthm}
\usepackage{mathrsfs}
\usepackage{amsbsy}
\usepackage{dsfont}
\usepackage{bbm}
\usepackage{wasysym}
\usepackage{stmaryrd}
\usepackage{url}

\input xypic
\xyoption{all}

\makeindex
\makeglossary

\begin{document}
\baselineskip = 16pt
\onehalfspacing
\newcommand \ZZ {{\mathbb Z}}
\newcommand \NN {{\mathbb N}}
\newcommand \RR {{\mathbb R}}
\newcommand \PR {{\mathbb P}}
\newcommand \AF {{\mathbb A}}
\newcommand \GG {{\mathbb G}}
\newcommand \QQ {{\mathbb Q}}
\newcommand \CC {{\mathbb C}}
\newcommand \bcA {{\mathscr A}}
\newcommand \bcC {{\mathscr C}}
\newcommand \bcD {{\mathscr D}}
\newcommand \bcF {{\mathscr F}}
\newcommand \bcG {{\mathscr G}}
\newcommand \bcH {{\mathscr H}}
\newcommand \bcM {{\mathscr M}}
\newcommand \bcI {{\mathscr I}}
\newcommand \bcJ {{\mathscr J}}
\newcommand \bcK {{\mathscr K}}
\newcommand \bcL {{\mathscr L}}
\newcommand \bcO {{\mathscr O}}
\newcommand \bcP {{\mathscr P}}
\newcommand \bcQ {{\mathscr Q}}
\newcommand \bcR {{\mathscr R}}
\newcommand \bcS {{\mathscr S}}
\newcommand \bcV {{\mathscr V}}
\newcommand \bcU {{\mathscr U}}
\newcommand \bcW {{\mathscr W}}
\newcommand \bcX {{\mathscr X}}
\newcommand \bcY {{\mathscr Y}}
\newcommand \bcZ {{\mathscr Z}}
\newcommand \goa {{\mathfrak a}}
\newcommand \gob {{\mathfrak b}}
\newcommand \goc {{\mathfrak c}}
\newcommand \gom {{\mathfrak m}}
\newcommand \gon {{\mathfrak n}}
\newcommand \gop {{\mathfrak p}}
\newcommand \goq {{\mathfrak q}}
\newcommand \goQ {{\mathfrak Q}}
\newcommand \goP {{\mathfrak P}}
\newcommand \goM {{\mathfrak M}}
\newcommand \goN {{\mathfrak N}}
\newcommand \uno {{\mathbbm 1}}
\newcommand \Le {{\mathbbm L}}
\newcommand \Spec {{\rm {Spec}}}
\newcommand \Gr {{\rm {Gr}}}
\newcommand \Pic {{\rm {Pic}}}
\newcommand \Jac {{{J}}}
\newcommand \Alb {{\rm {Alb}}}
\newcommand \Corr {{Corr}}
\newcommand \Chow {{\mathscr C}}
\newcommand \Sym {{\rm {Sym}}}
\newcommand \Prym {{\rm {Prym}}}
\newcommand \cha {{\rm {char}}}
\newcommand \eff {{\rm {eff}}}
\newcommand \tr {{\rm {tr}}}
\newcommand \Tr {{\rm {Tr}}}
\newcommand \pr {{\rm {pr}}}
\newcommand \ev {{\it {ev}}}
\newcommand \cl {{\rm {cl}}}
\newcommand \interior {{\rm {Int}}}
\newcommand \sep {{\rm {sep}}}
\newcommand \td {{\rm {tdeg}}}
\newcommand \alg {{\rm {alg}}}
\newcommand \im {{\rm im}}
\newcommand \gr {{\rm {gr}}}
\newcommand \op {{\rm op}}
\newcommand \Hom {{\rm Hom}}
\newcommand \Hilb {{\rm Hilb}}
\newcommand \Sch {{\mathscr S\! }{\it ch}}
\newcommand \cHilb {{\mathscr H\! }{\it ilb}}
\newcommand \cHom {{\mathscr H\! }{\it om}}
\newcommand \colim {{{\rm colim}\, }} 
\newcommand \End {{\rm {End}}}
\newcommand \coker {{\rm {coker}}}
\newcommand \id {{\rm {id}}}
\newcommand \van {{\rm {van}}}
\newcommand \spc {{\rm {sp}}}
\newcommand \Ob {{\rm Ob}}
\newcommand \Aut {{\rm Aut}}
\newcommand \cor {{\rm {cor}}}
\newcommand \Cor {{\it {Corr}}}
\newcommand \res {{\rm {res}}}
\newcommand \red {{\rm{red}}}
\newcommand \Gal {{\rm {Gal}}}
\newcommand \PGL {{\rm {PGL}}}
\newcommand \Bl {{\rm {Bl}}}
\newcommand \Sing {{\rm {Sing}}}
\newcommand \spn {{\rm {span}}}
\newcommand \Nm {{\rm {Nm}}}
\newcommand \inv {{\rm {inv}}}
\newcommand \codim {{\rm {codim}}}
\newcommand \Div{{\rm{Div}}}
\newcommand \CH{{\rm{CH}}}
\newcommand \sg {{\Sigma }}
\newcommand \DM {{\sf DM}}
\newcommand \Gm {{{\mathbb G}_{\rm m}}}
\newcommand \tame {\rm {tame }}
\newcommand \znak {{\natural }}
\newcommand \lra {\longrightarrow}
\newcommand \hra {\hookrightarrow}
\newcommand \rra {\rightrightarrows}
\newcommand \ord {{\rm {ord}}}
\newcommand \Rat {{\mathscr Rat}}
\newcommand \rd {{\rm {red}}}
\newcommand \bSpec {{\bf {Spec}}}
\newcommand \Proj {{\rm {Proj}}}
\newcommand \pdiv {{\rm {div}}}
\newcommand \NS{{\rm{NS}}}
\newcommand \wt {\widetilde }
\newcommand \ac {\acute }
\newcommand \ch {\check }
\newcommand \ol {\overline }
\newcommand \Th {\Theta}
\newcommand \cAb {{\mathscr A\! }{\it b}}

\newenvironment{pf}{\par\noindent{\em Proof}.}{\hfill\framebox(6,6)
\par\medskip}

\newtheorem{theorem}[subsection]{Theorem}
\newtheorem{conjecture}[subsection]{Conjecture}
\newtheorem{proposition}[subsection]{Proposition}
\newtheorem{lemma}[subsection]{Lemma}
\newtheorem{remark}[subsection]{Remark}
\newtheorem{remarks}[subsection]{Remarks}
\newtheorem{definition}[subsection]{Definition}
\newtheorem{corollary}[subsection]{Corollary}
\newtheorem{example}[subsection]{Example}
\newtheorem{examples}[subsection]{examples}

\title{Representability of codimension three cycles}
\author{Kalyan Banerjee}
\address{SRM University AP}

\email{kalyan.ba@srmap.edu.in}
\begin{abstract}
In this paper, we develop the notion of representability of co-dimension three cycles on a fourfold in terms of zero cycles modulo rational equivalence on surfaces.
\end{abstract}
\maketitle

\section{Introduction}
The representability question in the theory of Chow groups is an important question. Precisely, it means the following: let $X$ be a smooth projective variety over the field of complex numbers and let $\CH^p(X)$ denote the Chow group of co-dimension $p$ algebraic cycles on $X$ modulo rational equivalence. Let $\CH^p(X)_{\alg}$ denote the subgroup of $\CH^p(X)$ consisting of algebraically trivial cycles. We say that the group $\CH^p(X)_{\alg}$ is representable if there exists a smooth projective curve $C$ and a correspondence $\Gamma$ on $C\times X$ such that $\Gamma_*$ from $J(C)\cong \CH^1(C)_{\alg}$ to $\CH^p(X)_{\alg}$ is surjective.  The most interesting and intriguing case is the zero cycles on $X$.

The first breakthrough result in this direction is by Mumford, \cite{M}: the group $\CH^2(S)_{\alg}$ of a smooth projective complex algebraic surface with geometric genus greater than zero is not representable. Roitman \cite{R1} further generalized for higher-dimensional varieties proving that if the variety $X$ has a global holomorphic $i$-form on it, then the group $\CH^n(X)_{\alg}$ is not representable. Here $n$ is the dimension of $X$ and we have $0<i\leq n$. 

Then there is the famous converse question due to Spencer Bloch saying that: for a smooth projective complex algebraic surface $S$ with geometric genus equal to zero, the group $\CH^2(S)_{\alg}$ is representable. This question has been answered for the surfaces not of general type with geometric genus equal to zero by Bloch-Kas-Liebarman in \cite{BKL}. The conjecture is still open in general for surfaces of general type, but it has been solved in many examples in \cite{B}, \cite{IM}, \cite{GP}, \cite{PW}, \cite{V}, \cite{VC}.

In the case of highest codimensional cycles on a smooth projective variety $X$, the notion of representability can also be defined in another way. 

\begin{definition}\label{defn1}
We consider the natural map from the symmetric power $\Sym^d X$ to $\CH^n(X)_{\alg}$, for $d$ positive and $n$ being the dimension of $X$. Suppose that this map is surjective for some $d$ then we say that the group $\CH^n(X)_{\alg}$ is representable.
    
\end{definition} 

It can be proved as in [\cite{Vo}, Chapter 10, Section 1], that this notion of representability is equivalent to the first notion of representability introduced before.

So following the approach of Voisin as in \cite{Vo}, it is natural to ask whether there is a second notion of representability for lower co-dimensional cycles, specifically for co-dimension three cycles on a smooth projective fourfold. To do this, first, we assume that:

(I) There exist a smooth, projective variety $Y$ and a correspondence $\Gamma$ of codimension $p$ on $Y\times X$, such that $\Gamma_*$ is surjective from ${\CH_0(Y)}_{0}\cong {\CH_0(Y)}_{\alg}$ to $\CH^p(X)_0$, (here ${\CH_0(Y)}_0,\CH^p(X)_0$ denote the group of degree zero cycles on $Y,X$, respectively modulo rational equivalence).

Under this assumption, the group $\CH^p(X)_0$ is contained in $\CH^p(X)_{\alg}$. On the other hand, $\CH^p(X)_{\alg}$ is always contained in $\CH^p(X)_0$ (this fact is due to the existence of Chow varieties; for reference, see [\cite{Fulton}, Example 10.3.3]).

Now, the second notion of representability is as follows:
Consider the two-fold product of the Chow variety $C^p_{d}(X)\times C^p_{d}(X)$ where $C^p_d(X)$ is the projective variety that parameterizes the codimension $p$ and degree $d$ cycles on $X$. Note that the degree of a cycle is determined by fixing a embedding of $X$ into a projective space. Consider the natural map from this product to the group $\CH^p(X)_{0}$ and hence (under assumption) to $\CH^p(X)_{\alg}$ and ask the question: Does the surjectivity of a proper countable union (which is not a finite union) of these maps implies the representability in the first sense of $\CH^p(X)_{\alg}$ up to dimension two, i.e:
Let $X$ be a smooth projective variety and let $A^p(X)$ denote the algebraically trivial, codimension $p$ algebraic cycles on $X$, modulo rational equivalence, which was previously denoted by $CH^p(X)_{\alg}$ (for notational convenience). Then we say that $A^p(X)$ is weakly representable up to dimension two if there exist finitely many curves $C_1,\cdots,C_m$ with correspondences $\Gamma_1,\cdots,\Gamma_m$ on $C_1\times X, \cdots, C_m\times X$ and finitely many surfaces $S_1,\cdots,S_n$ with correspondences $\Gamma_j'$ on $S_j\times X$, such that
$$\sum_i \Gamma_i+\sum_j \Gamma_j'$$
is surjective from $\oplus_i A^1(C_i)\oplus _j A^2(S_j)$ to $A^p(X)$.

First, we prove the following in this direction:

\smallskip

Suppose that:

A) The Chow group of codimension $p$ cycles is generated by linear subspaces, that is, the natural map (Abel-Jacobi) from $\CH_0(F(X))$ to $\CH^p(X)$ is surjective.  Here $F(X)$ is the Fano variety of linear subspaces of codimension $p$. 

B) Suppose that $A^p(X)$ is representable in the sense that the map from a proper countable union of the twofold product of the Chow variety to $A^p(X)$ is surjective (which is not a finite union).

\begin{theorem}

 With the above assumptions, there exist a smooth projective surface $S$ and a correspondence $\Gamma$ to $S\times X$, such that
$\Gamma_*:A^2(S)\to A^p(X)$
is surjective. In fact $S$ is obtained as a hyperplane section of the Fano variety of co-dimension $p$-linear subspaces of $X$.

\end{theorem}

\smallskip

As an application, we show that $A^3(X)$ is REPRESENTABLE up to dimension two where $X$ is a smooth, projective Fano fourfold such that $A^3(X)$ is generated by differences of rational curves. This is  due to the result of \cite{TZ} over complex numbers.

\smallskip

Our argument in this direction is a careful analysis of the argument present in the approach of Roitman in \cite{R},  Voisin in [\cite{Vo}, Chapter 10, Section 1], where the authors deal with the case of zero cycles. First, we recall various notions of representability in the second sense, denoted as "finite dimensionality" of Chow groups of codimension $p$ cycles and show their equivalence. The key point is to use the Roitman's result on the map from the twofold product of the Chow variety to $A^p(X)$ saying that the fibers of this map are a countable union of Zariski closed subsets in the product of Chow varieties. Then after having these equivalent notions of "finite dimensionality" in hand, we proceed to the main theorem.

\smallskip


{\small Throughout the text, we work in the field of complex numbers.}

\section{Notion of representability of codimension three cycles}
Let $X$ be a smooth projective variety and let $A^i(X)$ denote the algebraically trivial, codimension $i$ algebraic cycles on $X$, modulo rational equivalence. Then we say that $A^i(X)$ is weakly representable up to dimension two if there exist finitely many curves $C_1,\cdots,C_m$ with correspondences $\Gamma_1,\cdots,\Gamma_m$ on $C_1\times X, \cdots, C_m\times X$ and finitely many surfaces $S_1,\cdots,S_n$ with correspondences $\Gamma_j'$ on $S_j\times X$, such that
$$\sum_i \Gamma_i+\sum_j \Gamma_j'$$
is surjective from $\oplus_i A^1(C_i)\oplus _j A^2(S_j)$ to $A^i(X)$. If we assume that $X$ is a fourfold, then the representability of $A^2(X)$ is a birational invariant. This is because if we blow up $X$ to $\wt{X}$, then $A^2(\wt{X})$ is isomorphic to $A^2(X)\oplus A^1(Z)$, where $Z$ is the center of the blow-up. Since $A^1(Z)$ is dominated by $J(\Gamma)$, for some smooth projective curve $\Gamma$, this will imply that if $A^2(X)$ is representable up to dimension two, then so is $A^2(\wt{X})$. Suppose that $X,Y$ are birational, such that $Y$ is obtained by one blow-up of $X$ and then one blow-down, then we have a generically finite map from $\wt{X}$ to $Y$, which gives a surjection at the level of $A^2$. So $A^2(X)$ representable up to dimension two implies the same for $A^2(\wt{X})$, hence the same for $A^2(Y)$. Changing the role of $X,Y$, we get the reverse implication.

Similarly 
\begin{theorem} The representability of $A^3(X)$ up to dimension two, where $X$ is a smooth projective fourfold is a birational invariant in $X$. 
\end{theorem}
\begin{proof}
This is because if we blow up $X$ along a surface or a curve then the blow-up formula gives us
$$A^3(\wt{X})=A^3(X)\oplus A^2(S)$$
or
$$A^3(\wt{X})=A^3(X)\oplus A^1(C)$$
where $S$ or $C$ is the center of the blow-up. So, if we blow up for many times, we are only adding $A^2$ of a surface or $A^1$ of a curve, so the representability up to dimension two remains.
\end{proof}

But this invariant of the representability of dimension two is not sufficient to find the obstruction to rationality of unirational fourfolds. This is because for unirational fourfolds, there always exists a dominant rational map from $\PR^4\dashrightarrow X$. By blowing up the indeterminacy of the rational map and observing that $A^3(\PR^4)$ is trivial, we see that $A^3(X)$ is representable up to dimension two by the previous argument. Therefore we need to go further to find a birational invariant which detects the non-rationality for a unirational smooth, projective fourfold. Suppose that there exists a birational map
$$\PR^4\dashrightarrow X$$
Then by the blow-up formula, we have
$$A^3(\widetilde{\PR^4})=\oplus_i A^2(S_i)\oplus \oplus_j A^1(C_j)$$
where $\widetilde{\PR^4}$ is the blow up along the indeterminacy locus consisting of $S_i$'s and $C_j$'s of the above mentioned birational map. It is well known that in characteristic zero any birational map is a sequence of blow-up and blow-downs. Let us assume that we get $X$ from $\PR^4$ by one blow up and one blow down. Then we have
$$A^3(\widetilde{\PR^4})=A^2(S)$$
or
$$=A^1(C)$$
depending on whether the center of the blow-up is a surface $S$ or a curve $C$. On the other hand $\widetilde{\PR^4}$ is also a blow-up of $X$ along a surface $S'$ or a curve $C'$. Therefore
$$A^3(\widetilde{\PR^4})=A^2(S')\oplus A^3(X)$$
or
$$=A^1(C')\oplus A^3(X)$$
and hence
$$A^2(S)=A^2(S')\oplus A^3(X)$$
or
$$A^2(S)=A^1(C')\oplus A^3(X)$$
or
$$A^1(C)=A^2(S')\oplus A^3(X)$$
or
$$A^1(C)=A^1(C')\oplus A^3(X)$$
So we see that $A^3(X)$ is not only representable up to dimension $2$ but the kernel of the map from $A^3(\widetilde{\PR^4})$ to $A^3(X)$ has kernel given by
$$A^2(S')$$
or $$A^1(C')\;.$$ So the kernel is representable up to dimension $2$ and in fact it is isomorphic to
$$A^2(S')$$
or $$A^1(C')\;.$$
Therefore considering the general case of many blow-ups and blow downs, we have that the group $A^3(X)$ is representable up to dimension two, that is there exists smooth projective curves and surface $C_i$'s and $S_j$'s and correspondences $\Gamma_i$'s and $\Gamma_j$'s on $C_i\times X$, $S_j\times X$, such that
$$\oplus_i \Gamma_i\oplus_j \Gamma_j: \oplus_i A^1(C_i)\oplus \oplus_j A^2(S_j)\oplus A^1(S_j)\to A^3(X)$$
is onto and  the kernel of this homomorphism is isomorphic to a direct sum
$$\oplus_l A^1(C_l)\oplus_m A^2(S_m)\oplus A^1(S_m)\;,$$
where $C_l, S_m$ are smooth projective curves and surfaces.
From now on, we fix this as the definition of representability up to dimension two. This is a birational invariant by the discussion above and supposing that $A^3(X)$ is not representable up to dimension two would lead to the non-rationality of $X$.

\subsection{Finite dimensionality up to dimension two in terms of Chow varieties}
Let $X$ be a smooth projective variety defined over the ground field $k$. Let $C^p_d(X)$ denote the Chow variety of $X$ parameterizing all codimension $p$ cycles on $X$ to a certain degree $d$. To consider the degree, we fix an embedding of $X$ into some projective space $\PR^N$. Consider the $k$-points of the variety $C^p_{d}(X)$. Then consider the map
$$\theta_p^d:C^p_d(X)\times C^p_d(X)\to \CH^p(X)$$
given by
$$(Z_1,Z_2)\mapsto [Z_1-Z_2]$$
where $[Z_1-Z_2]$ is the class of the cycle $Z_1-Z_2$ in the Chow group. Suppose that the algebraic and homological equivalence coincide on $X$ for codimension $p$ cycles. Then the map $\theta^p_d$ takes the values in $A^p(X)$. By using incorrect notation, we will denote the class $[Z_1-Z_2]$ as $Z_1-Z_2$. Now that
$$A^p(X)=\bigcup_d \im(\theta^p_d)$$
\begin{definition}
We say that $A^p(X)$ is finite dimensional up to dimension two if there exists a positive integer $n>1$ such that
$$A^p(X)=\bigcup_d \im(\theta^p_{nd})\;.$$
\end{definition}

\begin{example}
Consider a smooth projective surface $S$ of a geometric genus greater than zero. Then by Mumford's theorem it follows that $A^2(S)$ is not weakly representable, that is, there does not exist a smooth projective curve $C$ and a correspondence $\Gamma$ on $C\times S$ such that
$$\Gamma_*: J(C)\to A^2(S)$$
is onto. This has the consequence that none of the map $\theta^2_d$ from $\Sym^d S\times \Sym^d S$ is on. Now consider embedding $S$ in $\PR^n$, such that the genus of the hyperplane sections of $S$ is greater than $1$ (this can be achieved by appealing to the adjunction formula). Then by Bertini's theorem it follows that given any two closed points $P,Q$ on $S$, there exists a smooth hyperplane section through $P,Q$. This leads to the fact that
$$\oplus_t J(C_t)$$
surjects onto $A^2(S)$, here $C_t$ is a smooth hyperplane section of $S$. Now it is well known that
$$\Sym^g C_t\times \Sym^g C_t$$
surjects onto $J(C_t)$. So combining the above two facts we have that
$$\bigcup_d \theta^2_{gd}=A^2(S)$$
since $g>1$
$$\bigcup_d \Sym^{gd} S\times \Sym^{gd} S\subsetneq \bigcup_d \Sym^d S\times \Sym^d S\;.$$
Therefore it follows that $A^2(S)$ is finite dimensional up to dimension $2$.
\end{example}

\begin{example}
Let $X$ be a smooth cubic fourfold embedded in $\PR^5$ and let $F(X)$ denote the Fano variety of lines on $X$. It is well known that $A^3(X)$ is dominated by $A^4(F(X))$, that is, by zero cycles on $F(X)$. It is also known that $A^3(X)$, therefore $A^4(F(X))$ is not weakly representable. Therefore, there do not exist a curve $C$ and a correspondence $\Gamma$ on $C\times F(X)$ such that
$$\Gamma_*: J(C)\to A^4(F(X))$$
is onto. Therefore, there does not exist a single $\theta^4_d$, which is onto. By the argument as in the previous example, we have a very ample line bundle on $F(X)$,of high degree, therefore, the curves in this linear system are of high genus. Therefore we have
$$\bigcup_d \theta^4_{gd}=A^4(F(X))=A^3(X)\;.$$
Therefore the group $A^4(F(X))$ and therefore $A^3(X)$ are finite-dimensional up to dimension two.
\end{example}
Suppose that $A^p(X)$ is finite-dimensional up to dimension two. Then there exists $g>1$, such that
$$\bigcup_d \theta^p_{gd}$$
surjects onto $A^p(X)$. Then we define the dimension of $A^p(X)$ as
$$\min_g\max_d \dim(\im(\theta^p_{gd}))\;.$$
Here $$\dim(\im(\theta^p_{gd}))=2\dim(C^p_{gd}(X))-\dim(\theta^p_{gd})^{-1}(z)\;.$$
Suppose that this number
$$\min_g\max_d \dim(\im(\theta^p_{gd}))$$
is finite then we say that $A^p(X)$ has finite dimension. Now we prove this.
\begin{theorem}
Suppose that the linear subspaces of codimension $p$ on $X$ generate $A^p(X)$. Then $A^p(X)$ is finite dimensional up to dimension two if and only if it has finite dimension.
\end{theorem}
\begin{proof}
If $A^p(X)$ is finite dimensional up to dimension two then there exists $g>1$ such that
$$\bigcup_d \Sym^{gd}F(X)\times \Sym^{gd}F(X)$$
surjects onto $A^p(X)$, $F(X)$ is the variety of linear subspaces of codimension $p$ on $X$. Let us consider
$$R\subset (\bigcup_d \Sym^{gd}F(X)\times \Sym^{gd} F(X))\times (\bigcup_d \Sym^{g'd}F(X)\times \Sym^{g'd}F(X))$$
consisting of 4-tuples
$$(Z_1,Z_2,Z_1'.Z_2')$$
such that
$$\theta^p_{gd}(Z_1,Z_2)=\theta^p_{g'd}(Z_1',Z_2')$$
as we have
$$\bigcup_d \theta^p_{gd}=A^p(X)$$
we have $R$ surjecting onto the second co-ordinate
$$\bigcup_d \Sym^{g'd}F(X)\times \Sym^{g'd}F(X)\;.$$
Also it follows by the Mumford-Roitman argument that $R$ is a countable union of Zariski closed subsets in the product of the infinite union. Thus by the Baire category argument (that is over an uncountable ground field any irreducible variety cannnot be obtained as a countbale union of its proper subvarieties), there exists $R_d$ such that it surjects onto
$$\Sym^{g'd}F(X)\times \Sym^{g'd}F(X)$$
in particular $\dim(R_d)\geq 2ng'd$
where $n$ is the dimension of $F(X)$. Now if we consider the map from
$$R_d\to \Sym^{gd'}F(X)\times \Sym^{gd'}F(X)$$
then the dimension of the fiber of this map has dimension greater or equal than
$$\dim(R)-2ngd\geq 2ng'd-2ngd'$$
Note that this algebraic set
$$R_0\cap (Z_1,Z_2)\times \Sym^{g'd}F(X)\times \Sym^{g'd}F(X)$$
is contained in
$$(\theta^{p}_{g'd})^{-1}(\theta^p_{gd'}(Z_1,Z_2))$$
Therefore the dimension of the fibers of $\theta^p_{g'd}$ is greater or equal than $$2ng'd-2ngd'$$
therefore
$$\dim(\im(\theta^p_{g'd}))\leq 2ngd'$$
hence the minima
$$\max_d\dim(\im(\theta^p_{g'd}))\leq \max_{d'}2ngd'$$
whence
$$\min_{g'}\max_d \dim(\im(\theta^p_{g'd}))\leq \max_{d'}2ngd' $$
So we conclude that $A^p(X)$ is of finite dimension.
\medskip
Conversely suppose that $A^p(X)$ is of finite dimension. Then for large $d$ the number
$$\min_g\max_d \im(\theta^p_{gd})$$
is constant. We have to show that there exists $g$ greater than one such that
$$\bigcup_d \im(\theta^p_{gd})$$
surjects onto $A^p(X)$. Now consider a positive integer $m$, such that
$$\dim(\theta^p_{(g+m)d})=\dim(\theta^p_{gd})$$
for large $g$, also there exists $m$ such that
$$(g+m)d>gd\;.$$
Let $x$ be a closed point in $F(X)$ and consider the embedding of
$$\Sym^{gd}F(X)\times \Sym^{gd}F(X)$$
into
$$\Sym^{(g+m)d}F(X)\times \Sym^{(g+m)d}F(X)$$
given by
$$(Z_1,Z_2)\mapsto (Z_1+lx,Z_2+lx)$$
here $l=(g+m)d-gd$ and denote the map by $i_x$. Then we have
$$\theta^p_{(g+m)d}\circ i_x=\theta^p_{gd}$$
Let $F$ be the general fiber of $\theta^p_{gd}$. Let $F'$ be the general fiber of $\theta^p_{(g+m)d}$. Now we have
$$\dim(\im(\theta^p_{gd}))=2ngd-\dim(F)$$
and
$$\dim(\im(\theta^p_{(g+m)d}))=2n(g+m)d-\dim(F')$$
Since by our assumption
$$\dim(\im(\theta^p_{gd}))=\dim(\im(\theta^p_{(g+m)d}))$$
we have
$$2ngd-\dim(F)=2n(g+m)d-\dim(F')$$
so
$$\dim(F')=\dim(F)+2n(g+m)d-2ngd\;.$$
Now consider $F''$ to be the fiber of $\theta^{p}_{(g+m)d}$ through $(Z_1+lx,Z_2+lx)$, for a general point $(Z_1,Z_2)$. By the semicontinuity theorem for dimension we have
$$\dim(F'')\geq \dim(F')\;.$$
Now consider the set $R$ as before inside
$$\Sym^{(g+m)d}F(X)\times \Sym^{(g+m)d}F(X)\times \Sym^{gd}F(X)\times \Sym^{gd}F(X)$$
consisting of $$(Z_1,Z_2,Z_1',Z_2')$$
such that
$$\theta^p_{(g+m)d}(Z_1,Z_2)=\theta^p_{gd}(Z_1',Z_2')$$
Given $(Z_1,Z_2)$ in $\Sym^{gd}F(X)\times \Sym^{gd}F(X)$, there exists $(Z_1+lx,Z_2+lx,Z_1,Z_2)$ in $R$. So the map from $R$ to $\Sym^{gd}F(X)\times \Sym^{gd}F(X)$ is surjective. Since $R$ is a countable union of Zariski closed subsets in the product, by the uncountability of the ground field there exists an irreducible component $R_0$ of $R$ mapping dominantly onto $\Sym^{gd}F(X)\times \Sym^{gd}F(X)$ also it passes through $(Z_1+lx,Z_2+lx)$ for a general $(Z_1,Z_2)$. The fiber of  the projection from $R_0$ to $\Sym^{gd}F(X)\times \Sym^{gd}F(X)$ is of dimension equal to $\dim(F'')$. So the fiber dimension is greater or equal than
$$\dim(F)+2n(g+m)d-2ngd$$
Now consider the projection $$p:R_0\to \Sym^{(g+m)d}F(X)\times \Sym^{(g+m)d}F(X)$$
The dimension of
$$\dim(p(R_0))\geq \dim (R_0)-\dim(F)$$
the above is greater or equal than
$$\dim(F'')+2ngd-\dim(F)\geq \dim(F)+2ngd+2n(g+m)d-2ngd=2n(g+m)d\;.$$
So $p$ is surjective and hence $\im(\theta^p_{(g+m)d})=\im\theta^p_{gd}\;.$ From this it follows that
$$\bigcup_{d''}\im(\theta^p_{(g+m)dd''})=\bigcup_d(\im(\theta^p_{gdd''}))$$
Therefore we have that
$$\bigcup_d''\im(\theta^p_{gdd''})=A^p(X)\;.$$
\end{proof}
Now we prove the following main theorem of this section:
\begin{theorem}
Suppose that $X$ is a smooth, projective fourfold on which the algebraic and homological equivalence coincide. Then $A^3(X)$ is representable up to dimension $2$ if and only if it is finite dimensional up to dimension $2$.
\end{theorem}
\begin{proof}
Consider the map $\theta^3_d$ from $C^3_{d}(X)\times C^3_d(X)$ to $A^3(X)$ given by
$$(Z_1,Z_2)\mapsto Z_1-Z_2$$
 Since $A^3(X)$ is actually generated by linear subspaces, the above map restricted to $\Sym^d F(X)\times \Sym^d F(X)$ is surjective, continue to call it $\theta^3_d$. Let us consider large $g$ such that $\dim(\cup_d \im(\theta^3_{gd})$ is constant and equal to $K$. Then for a general fiber we have
$$\max_d(\dim(\Sym^{gd} F(X))-\dim((\theta^3_{gd})^{-1}(z))=K$$
So for all $d$ such that
$$\dim(\Sym^{gd}F(X))-\dim((\theta^3_{gd})^{-1}(z))\geq K$$

Now the group $A^4(F(X))$ is generated by zero cycles supported on hyperplane sections of $F(X)$ (to achieve the hyperplane sections we fix an embedding of $F(X)$ into some projective space). Let $Y_t$ be a smooth hyperplane  section of $F(X)$. Then the image of the group $A_0(Y_t)$ under the homomorphism
$$A_0(F(X))\to A^3(X)$$
is also finite dimensional up to dimension $2$ in the sense that it is dominated by
$$\bigcup_d \im(\theta^3_{gd})\;.$$

Now we prove that there exists a $g$ and $d_0$ depending on $g$ such that an irreducible component $Z$ of maximal dimension of a general fiber of $$\theta^3_{gdd_0}|_{\Sym^{gdd_0}Y_t}$$ for all $d$, cannot be contained in a set of the form
$$\Sym^{gdd_0-gid_0}F(X)+ W$$
where $W$ is in $\Sym^{gid_0} F(X)$, $\dim W<id_0$ and the above $+$ means the image of the natural map from
$$\Sym^{gdd_0-gid_0}F(X)\times W\to \Sym^{gdd_0}F(X)\;.$$
If possible, assume that $Z$ is contained in such a set. Note that the dimension of $Z$ is greater than  $\dim(\Sym^{gdd_0}F(X))-K$. So we have
$$\dim(\Sym^{gdd_0-gid_0}F(X))\geq \dim(\Sym^{gdd_0} F(X))-K-id_0+1\;.$$
Let the dimension of $F(X)$ be $n$. Then the above says
$$n(gdd_0-gid_0)\geq ngdd_0-K-id_0+1$$
which implies
$$i<K/ngd_0-d_0\;.$$
Consider the subset
$$Z'=\{(z,w)|z+w\in Z\}\subset \Sym^{gdd_0-gid_0}F(X)\times W\;.$$
By definition this set dominates $Z$, hence is of dimension greater or equal than $ngdd_0-K$. So the general fibers of the second projection $\pr_2:Z'\to W$ are of dimension atleast
$$ngdd_0-K-id_0+1\;.$$
Also note that
$$\theta^p_{gdd_0}(z+w)=\theta^p_{gdd_0-gid_0}(z)+\theta^p_{gid_0}(w)\;.$$
Since $\theta^p_{gdd_0}$ is constant along $Z$, we have $\theta^p_{gdd_0-gid_0}$ is constant along $Z'_w$. Thus if $Z$ passes through a very general point of $\Sym^{gdd_0} F(X)$, then $Z'_w$ passes through a very general point of $\Sym^{gdd_0-gid_0}F(X)$. So we have $\dim(Z'_w)$ is less than the dimension of a general fiber of $\theta^p_{gdd_0-gid_0}$ for generic $w$.  Now the dimension of $Z'_w$ is greater than or equal to
$$ngdd_0-K-id_0+1$$
the dimension of the fiber of $\theta^p_{gdd_0-gid_0}$ is  equal to $(gdd_0-gid_0)n-K$ (because we could choose $(d-i)$ strictly greater than $d-K/ngd_0-d_0$ and arbitrarily large and appeal to finite dimensionality also we need to choose a $g=g(d-i)$ to attain the equality of the dimension of the fiber also we may have to replace $d_0$ by a multiple of $d_0$) and therefore
$$ngdd_0-K-id_0+1\leq ngdd_0-gnid_0-K$$
which implies
$$(ng-1)id_0+1\leq 0$$
which is absurd.
Let us assume that $n\geq 2$, and we have $ngdd_0-K\geq gdd_0$. Consider the following lemma.

\begin{lemma}
\label{lemma1}
Let $Y$ be an ample hypersurface of $F(X)$ and let $Z$ be an irreducible subset of $\Sym^{gd} F(X)$ not contained in any subset of the form
$$\Sym^{gd-gi}F(X)+W$$
with $W\subset \Sym^{gi} F(X)$ and dimension of $W$ is strictly less than $gi$. Then $Z$ intersects $\Sym^{gd} Y$, provided that $\dim(Z)\geq gd$.
\end{lemma}
Therefore by applying the lemma we see that a general fiber of $\theta^p_{gdd_0}|_{\Sym^{gdd_0}Y_t}$ intersects $\Sym^{gd} Y$, for a fixed $g$ and provided that $n\geq 2$. Hence $\theta^p_{gdd_0}|_{\Sym^{gdd_0}Y_t}$ and $\theta^p_{gd}|_{\Sym^{gdd_0} Y_t\cap \Sym^{gdd_0}Y}$ have same image and the later has image of bounded dimension. So we can apply the lemma again and finally get that $\theta^p_{gdd_0}|_{\Sym^{gdd_0}Y_t}$ and $\theta^p_{gdd_0}|_{\Sym^{gdd_0} C_t}$ have same image, where $C_t$ is a smooth projective curve obtained by intersecting $n-1$ many ample hypersurfaces. Therefore there exists a surface $S$ in $F(X)$, such that
$$A_0(S)\to A^3(F(X))$$
is surjective.

\smallskip

\textit{Proof of Lemma \ref{lemma1}}: Consider the quotient map $r:F(X)^d\to \Sym^d F(X)$. Let $r^{-1}(Z)=\wt{Z}$, let $\wt{Z_0}$ be a component of $\wt{Z}$ dominating $Z$. By the hypothesis we have the following:

for every $i\geq 1$ and every subset $I$ of cardinality $i$, we have $\dim p_I(\wt{Z_0})\geq i$, where $p_I$ is the projection from $F(X)^d$ to $F(X)^i$ corresponding to the set of indices. Since $\wt{Z_0}$ dominates $Z$, it is sufficient to prove that $\wt{Z_0}$ intersects $Y^d$, for an ample hyper-surface $Y $  in $F(X)$. Consider a complete intersection $V$ in $\wt{Z_0}$, which is obtained by intersecting $\wt{Z_0}$ with finitely many ample hyper-surfaces. So that dimension of $V$ is $d$. Then the hypotheses on $\wt{Z_0}$ implies that same would be true for  $V$. So without loss of generality we can assume that dimension of $\wt{Z_0}=d$. Consider a de-singularization $Z'$ of $\wt{Z_0}$. Consider the divisors
$$D_i:=(\pr_i\circ \tau)^{-1}(Y)$$
where $\tau$ is the natural map from $Z'$ to $\wt{Z_0}$. Now $Y$ is ample, so we have
$$(\pr_i\circ \tau)_*((\pr_i\circ\tau)^*(Y).C)=Y.(\pr_i\tau)_*C\geq 0$$
which means that $D_i$'s are numerically effective. Our claim will follow from the fact that $D_1\cap\cdots\cap D_d$ is non-empty. So we prove that $D_1\cap\cdots\cap D_d$ is non-empty. First suppose that $d=2$. We have $D_1,D_2$ two divisors numerically effective. Hence we have
$$D_1^2\geq 0;, \quad D_2^2\geq 0$$
Suppose that $D_1.D_2=0$. Then the intersection matrix of $(D_1,D_2)$ is semipositive. So by the Hodge index theorem we have $D_1=rD_2$ for some integer $r$. Hence $(D_1+D_2)^2=0$.  But $D_1+D_2$ is the pull-back of an ample divisor on $F(X)\times F(X)$, under a generically finite map. So it is ample. Therefore $(D_1+D_2)^2>0$, which is a contradiction. The general case follows from this.

\end{proof}

\end{document}